\documentclass[leqno]{amsart}
\usepackage{amsmath,amsthm, mathrsfs, amssymb}
\usepackage{color}
\newtheorem{remark}{Remark}[section]
\usepackage{hyperref}
\hypersetup{
    colorlinks=true,
    linkcolor=blue,
    filecolor=magenta,      
    urlcolor=cyan,
    pdftitle={Sharelatex Example},
}
\numberwithin{equation}{section}

\title[Zakharov-Rubenchik]{On the Cauchy problem for the Zakharov-Rubenchik/ Benney-Roskes system}
\author[H. Luong]{Hung Luong}
\address{Fak. Mathematik, University of Vienna, Oskar MorgensternPlatz 1\\
A-1090 Wien, Austria}

\email{luongh88@univie.ac.at}

\author[N. Mauser]{Norbert Mauser}
\address{Wolfgang Pauli Institute c/o Fak. Math. Univ. Vienna, Oskar MorgensternPlatz 1\\
A-1090 Wien, Austria}
\email{ norbert.mauser@univie.ac.at}

\author[J.-C. Saut]{Jean-Claude Saut}
\address{Laboratoire de  Math\' ematiques, UMR 8628\\
Universit\' e Paris-Saclay, Paris-Sud and CNRS\\ F-91405 Orsay, France}
\email{jean-claude.saut@u-psud.fr}

\date{February 16th,  2018}

\begin{document}
	%Command definition.
	\newcommand{\R}{\mathbb{R}}
	\newcommand{\T}{\mathbb{T}}
	\newcommand{\Z}{\mathbb{Z}}
	\newcommand{\C}{\mathcal{C}}
	\newcommand{\U}{\mathbb{U}}
	\newcommand{\N}{\mathbb{N}}
	\newcommand{\W}{\mathcal{W}}
	\newcommand{\s}{\text{sinh}}
	\newcommand{\norm}[2]{ \left \lVert #1; \, #2  \right \rVert}
	\newcommand{\norma}[1]{ \left \lVert  #1 \right \rVert}
	\newcommand{\bra}[1]{\left \langle #1 \right \rangle}
	%\DeclareRobustCommand{\S}{sech }

\newtheorem{theorem}{Theorem}[section]
\newtheorem{lemma}{Lemma}[section]
\newtheorem{proposition}{Proposition}[section]
\newtheorem{definition}{Definition}[section]

\maketitle
\tableofcontents
\begin{center}
Dedicated to Vladimir Georgiev with admiration and friendship.
\end{center}
% Main part

\begin{abstract}
We address various issues concerning the Cauchy problem for the Zakharov-Rubenchik system, (known as the Benney-Roskes system in water waves theory) which models the interaction of short and long waves in many physical situations. Motivated by the transverse stability/instability of the one-dimensional solitary wave (line solitary), we study the Cauchy problem in the background of a line solitary wave. 
%We also prove a {\it long time } existence of solutions to the Cauchy problem of the Benney-Roskes system, a crucial step in the full justification of the system (see \cite{La1}).
 
\end{abstract}
\section{Introduction}
\label{section introduction}

This paper is concerned with various issues concerning the Cauchy problem for the two or three-dimensional  Zakharov-Rubenchik (or Benney-Roskes) system and its perturbation by a line soliton. The Zakharov-Rubenchik system  is by no doubts  a fundamental one, being a  "generic" asymptotic system  in the so-called modulation regime (slowly varying envelope of a fast oscillating train) and it was actually derived in various physical contexts. Moreover it contains in various limits the classical (scalar)  Zakharov system (coupling a nonlinear Schr\"{o}dinger equation and a wave equation, see \eqref{Zakh} below) and the Davey-Stewartson systems (coupling a nonlinear Schr\"{o}dinger equation and an elliptic equation). We refer to  \cite{Za-Ku} for more details on the formal derivation of those systems and on the physical background. 

The Davey-Stewartson system was first derived formally in the context of water waves in \cite{DS, AS, DR} (see also \cite{Co, CoLa} for a derivation of Davey-Stewartson systems in a different context). However, as noticed in \cite{La1} it is less general than the Benney-Roskes system \eqref{Benney-Roskes model} below in the sense that the initial conditions for the acoustic type components have to be prepared to obtain an approximation of the full water waves system. 

We refer to  \cite{Cor} for a rigorous justification of the Zakharov  limit of the Zakharov-Rubenchik system and to \cite{O2} for the Schr\"{o}dinger limit of the Zakharov-Rubenchik system in the one-dimensional case and for well-prepared initial data.

 The Zakharov-Rubenchik/Benney-Roskes system  is thus richer than
those simpler models and should capture more of the original dynamics. It was  introduced in \cite{ZR} (see also the survey article \cite{Za-Ku}) to describe the interaction of spectrally narrow high-frequency wave packet of small amplitude with a low-frequency acoustic type oscillations. The analysis is general
and carried out in the Hamiltonian formalism and yields the following {\it universal} system

\begin{equation}
 \label{ZR-system}
  \left \{
   \begin{aligned}
   & \psi_t +v_g \psi_x +i\frac{\omega"}{2} \psi_{xx}+ i\frac{v_g}{2k}\Delta_\perp \psi- i(q|\psi|^2 + \beta \rho+\alpha \phi_x) \psi = 0, \\
   & \rho_t + \rho_0\Delta \phi + \alpha (|\psi|^2)_x =0 \\
   & \phi_t + \frac{c^2}{\rho_0} \rho +\beta |\psi|^2 = 0,
   \end{aligned}
  \right.
\end{equation}
where $v_g, \, \omega", \, k, \, q, \, \beta, \, \alpha, \rho_0, \, c$ are parameters .The two last equations describe the acoustic type waves and $\Delta_{\perp}= \partial_x^2+\partial_y^2$ or $\partial_x^2,$  $\Delta=\Delta_{\perp}+\partial_x^2.$

In two space dimensions a more specific (formal) derivation in the context of surface water waves is displayed in \cite{BR}  and rigorously justified in \cite{La1}, see below for a more precise description. \\

\vspace{0.3cm}
 In the notations of \cite {PoSa} (see also  \cite{PSS} where it is used in the context of Alfv\' en waves in dispersive MHD), the Zakharov-Rubenchik system has the form
\begin{equation}
 \label{2-D Z-R system 1}
  \left \{
   \begin{aligned}
   & \psi_t - \sigma_3 \psi_x - i \delta \psi_{xx} - i \sigma_1 \Delta_{\perp}\psi + i \left \{  \sigma_2 |\psi|^2 +W (\rho+ D\phi_x) \right \} \psi = 0 \\
   & \rho_t + \Delta \phi + D (|\psi|^2)_x =0 \\
   & \phi_t + \dfrac{1}{M^2} \rho + |\psi|^2 = 0,
   \end{aligned}
  \right.
\end{equation}
where $\psi: \R \times \R^d \rightarrow \mathbb{C}$, $\rho,\phi : \R \times \R^d \rightarrow \R, d=2,3$ describe  the fast oscillating and respectively acoustic type waves.

Here  $\sigma_1, \sigma_2, \sigma_3= \pm 1,$ $W>0$ measures the strength of the coupling with acoustic type waves, $M>0$ is a Mach number, $D\in \R$  is associated
to the Doppler shift due to the medium velocity and $\delta \in \R$ is a nondimensional dispersion coefficient.

%The physical meanings of the parameters $\rho_1, \rho_2, \rho_3, \delta, \C, \W$ is given in \cite{PSS}. We have $\rho_1, \rho_2, \rho_3 \in \{ \pm 1 \}$, $\delta, \C, M \in \R$ and $W \in \R_+$.\\

When $\alpha=0$ (resp. $D=0$) in \eqref{ZR-system} (resp. \eqref{2-D Z-R system 1}) the Zakharov-Rubenchik system reduces to the classical (scalar) Zakharov system (see {\it eg} Chapter V in \cite{SuSu}). More precisely, in the framework of \eqref{ZR-system}, one gets

\begin{equation}
 \label{Zakh}
  \left \{
   \begin{aligned}
   & \psi_t +v_g \psi_x +i\frac{\omega"}{2}  \psi_{xx}+ i\frac{v_g}{2k}\Delta_\perp \psi= i(q|\psi|^2 + \beta \rho) \psi, \\
   & \rho_{tt}-c^2\Delta \rho=\beta \rho_0\Delta |\psi|^2,
   %& \phi_t + \frac{c^2}{\rho_0} \rho +\beta |\psi|^2 = 0,
   \end{aligned}
  \right.
\end{equation}

which is a form of the $2$ or $3D$ Zakharov system. Note however that the second order operator in the first equation is not necessarily elliptic.

\vspace{0.3cm}
The local well-posedness in $H^s(\R^d) \times H^{s-1/2}(\R^d) \times H^{s+1/2}(\R^d)$ with $s>\frac{d}{2}, d=2,3$ for \eqref{2-D Z-R system 1}, \eqref{ZR-system} was obtained in \cite{PoSa} by using the local  smoothing property of the free Schr\"odinger operator after reducing the system to a quasilinear (non local) Schr\"{o}dinger equation. Since it uses dispersive properties of the free Schr\"{o}dinger group that are valid only in the whole space the proof does not extend to the Cauchy problem posed in $\T^d$ or $\R^{d-1}\times \T,$ the later situation being relevant for transverse stability issues. On the other hand, when applied to the Benney-Roskes system \eqref{notations B-R/Z-R equation} below, it provides an existence time of order $O(1)$\footnote{Roughly speaking, the idea in \cite{PoSa} is to reduce the system to a (nonlocal) quasilinear Schr\"{o}dinger equation. When $\epsilon$ is taken into account, the crucial dispersive smoothing estimate on the Schr\"{o}dinger group has a $1/\epsilon$ factor while  the nonlinear term has a $\epsilon$ factor.} while an existence time of order $O(1/\epsilon)$ is needed to fully justify the Benney-Roskes as a water wave model on the correct time scales (see \cite{La}).

 Local well-posedness of the Zakharov-Rubenchik/Benney-Roskes system was  also obtained  in \cite{Ob}, for $s > 2 $ with the additional condition $\delta \sigma_1 > 0$ (that is the second order operator in the first equation of \eqref{2-D Z-R system 1}, \eqref{ZR-system}  is elliptic) by using an energy method inspired by the work of Schochet-Weinstein in \cite{SW} on the nonlinear Schr\"{o}dinger limit of the Zakharov system. The method used in \cite{Ob} and \cite{SW} consists in rewriting the Zakharov system (or the Zakharov-Rubenchik system) as a dispersive (skew-adjoint) perturbation of a symmetric nonlinear hyperbolic system and  it uses only the algebraic structure of the system. A shortcoming of the method is that one has to prepare the initial data.
 
 We will see that, when the small parameter $\epsilon$ is included, this method provides also the existence on the  time scale $O(1)$ in the context of water waves (see the Benney-Roskes system \eqref{Benney-Roskes model} below) and moreover that  it can be applied to the system obtained from  \eqref{2-D Z-R system 1}  which is satisfied by a (localized) perturbation of a line soliton. Also, since it does not use any dispersive property of the Schr\"{o}dinger group, it applies  to the Cauchy problem in $\T^d$ or $\R^{d-1}\times \T,$ a situation that has not been addressed before (see on the other hand \cite{Bo, BC} for the periodic Zakharov system) .

 Thus, none of the two aforementioned methods seems to give the expected existence time scale  for the Benney-Roskes system. Nevertheless they provide different results for Zakharov-Rubenchik type systems. The "dispersive method" used in \cite{PoSa} works only in $\R^d$ but does not need the Schr\"{o}dinger part of the system to be "elliptic" (that is it does not need the condition $\delta \sigma_1 > 0$). Also it lowers the regularity on the initial data (an effect of the dispersive smoothing effect) and could be applied as well to (possibly non physical) nonlinear perturbations of the system.
 
% provides the non optimal existence time $1/\sqrt \epsilon.$  Also this method does not extend straightforwardly to the system posed on $\R^d\times \T,$ a natural setting to study the transverse stability properties of line solitary waves with respect to {\it y-periodic} perturbations. 
 On the other hand, the Schochet-Weinstein type, "hyperbolic like" methods allow to deal with the periodic or semi-periodic cases, but are relatively rigid (they rely on the algebraic structure of the system) and require initial data in the "hyperbolic space" $H^s(\R^d), s>\frac{d}{2}+1.$ .

The situation is better understood in spatial dimension one. Oliveira \cite{O} proved the local (thus global using the conservation laws below) well-posedness in $H^2(\R)\times H^1(\R)\times H^1(\R).$ This result was improved in \cite {LM} where in particular global well-posedness was established in the energy space $H^1(\R)\times L^2(\R) \times L^2(\R).$ 

%On the other hand we do not know of previous well-posedness results for the Zakharov-Rubenchik system in a periodic or semi-periodic setting (see for instance \cite{Bo} for the one-dimensional periodic Cauchy problem for the Zakharov system).
\vspace{0.3cm}
It is worth noticing that \eqref{2-D Z-R system 1} possesses two conserved quantities,
the $L^2$ norm 

$$\int |\psi(x,y_\perp,t)|^2=\int |\psi(x,y_\perp,0)|^2$$

where $y_\perp=y$ or $(y,z),$
and, after the change of variable $(x,t)\to (x+\sigma_3t,t),$ the Hamiltonian

\begin{equation}\label{Hamil}
\begin{aligned}
E(t) & =  \int_{\R^d}\left(\frac{\delta}{2}|\psi_x|^2+\frac{\sigma_1}{2}|\nabla_\perp \psi|^2+\frac{\sigma_2}{4}|\psi|^4+\frac{W}{4M^2}\rho^2+\frac{W}{4}|\nabla\phi|^2+\frac{W}{2}(\rho+D\phi_x)|\psi|^2\right) \\
&  =E(0)
\end{aligned}
\end{equation}

The conservation laws are used in  \cite{PoSa} to obtain global weak solutions under suitable assumptions on the coefficients. We will use them in Section 5 to prove the global existence of weak solutions of the systems obtained by perturbing a line (dark) soliton. Note also that the conservation laws can be used to get the global well-posedness of the Zakharov-Rubenchik, Benney-Roskes system in space dimension one (see \cite{O}).

\vspace{0.5cm}
As aforementioned, in the context of water waves, the Zakharov-Rubenchik system is known as the  Benney-Roskes sytem and it was formally derived in \cite{BR}. We follow here the notations in \cite{La1}, where a rigorous derivation is performed. \\
%The standard Benney-Rokes model was originally derived in [25] for the water waves problem and previously derived by Zakharov and Rubenchik [334] for acoustic waves - see also [335]. If we use the following notations
\begin{equation}
 \label{notations B-R/Z-R equation}
 \begin{split}
  & \textbf{k} = |\textbf{k}| \textbf{e}_x, \quad \omega(\textbf{k}) = \underline{\omega}(|\textbf{k}|), \\
  & \omega = \underline{\omega}(\textbf{k}), \quad \omega' = \underline{\omega'}(|k|), \quad \omega'' = \underline{\omega''}(|\textbf{k}|), \\
  & \text{where} \quad \underline{\omega}(\xi)=  \left( \left( g + \dfrac{\sigma}{\rho} |\xi|^2  \right) |\xi| \tanh (\mu |\xi|)  \right)^{1/2}  
 \end{split}
\end{equation}
is the dispersion relation  of water waves and
where $|\textbf{k}|$ is a fixed wave number, $g$ is the gravity, $\sigma\geq 0$ is a surface tension coefficient, $\rho$ is the density of the water and $\mu$ is the shallowness parameter (square of   the typical fluid depth over a typical horizontal scale)  which is large or infinite in the deep water models and $\alpha = -\dfrac{9}{8 \sigma^2} (1- \sigma^2)^2$.\\

The small parameter $\epsilon$ is the {\it wave steepness} that is the ratio of a typical amplitude of the wave over a typical horizontal scale. Recall (\cite{La1}) that the typical time scale for the solutions of \eqref{Benney-Roskes model} below is $1/\epsilon $  and so it is crucial to establish the well-posedness on those time scales.

The Benney-Roskes equations can then be written in 2 dimensions  as follows
\begin{equation}
\label{Benney-Roskes model}
 \left \{
  \begin{aligned}
   & \partial_t \psi_{01} + \omega' \partial_x \psi_{01} - i \epsilon \dfrac{1}{2} (\omega'' \partial_x^2 + \dfrac{\omega'}{|\textbf{k}|} \partial_y^2) \psi_{01} \\
   & \qquad +  \epsilon i \left(  |\textbf{k}| \partial_x \psi_{00} + \dfrac{|\textbf{k}|^2}{2 \omega} (1 - \sigma^2) \zeta_{10} + 2 \dfrac{|\textbf{k}|^4}{\omega} (1- \alpha) |\psi_{01}|^2    \right) \psi_{01} = 0 \\
   & \partial_t \zeta_{10} + \sqrt{\mu} \Delta \psi_{00} = - 2 \omega |\textbf{k}| \partial_x (|\psi_{01}|^2) \\
   & \partial_t \psi_{00} + \zeta_{10} = - |\textbf{k}|^2 (1- \sigma^2) |\psi_{01}|^2.
  \end{aligned}  
 \right.
\end{equation}

It is known   (see {\it eg} \cite{La1}  Chapter 8) that $\omega ' >0$, while for purely gravity waves ($\sigma =0)$  $\omega$ is a concave function, thus $\omega''<0$ and the Schr\"{o}dinger equation in the Benney-Roskes system is "non elliptic".

On the other hand, in presence of surface tension, the condition $\omega''>0$ is possible as shown in the following computation.

 %We show that by choosing a  suitable wave number $|\textbf{k}|$, $\omega'$ and $\omega''$ will have the same sign.\\
  For simplicity of notations, we will consider $\omega$ of the following form instead of \eqref{notations B-R/Z-R equation}
 \[
  \omega(r) = \left( (1+ \gamma r^2) r \tanh(\mu r)  \right)^{1/2}
 \]
with $\gamma > 0$  depends on $(\, g, \, \rho),$ is proportional to $\sigma$ and  $r = |\textbf{k}|$. We have
\[ 
 \begin{aligned}
 \omega'(r) = & \left(     (1+ 3 \gamma r^2) \tanh(\mu r) + \mu (r+ \gamma r^3) \,  \text{sech}^2 (\mu r)       \right) \\
 & \times \dfrac{1}{2} \left(  (r+ \gamma r^3) \tanh(\mu r)  \right)^{-1/2}
 \end{aligned}
\]
 and 
 \[
  \begin{aligned}
 &  \omega''(r)  \\
  & \quad = - \dfrac{1}{4} \left( (1+ 3\gamma r^2) \tanh (\mu r) + \mu  (r+ \gamma r^3) \text{sech}^2(\mu r) \right)^2 \, \left( (r+\gamma r^3) \tanh(\mu r) \right)^{-3/2} \\
  & \quad \quad + \dfrac{1}{2} \left( (r+\gamma r^3) \tanh (\mu r) \right)^{-1/2} \, \bigg ( 6\gamma r \tanh (\mu r) + 2 \mu (1+3\gamma r^2) \text{sech}^2(\mu r) \\
  & \quad \qquad -2 \mu^2(r+\gamma r^3) \tanh (\mu r) \text{sech}^2(\mu r) \bigg ).
  \end{aligned}
 \]
 We see that $\omega'(r) > 0$ with $\gamma , r > 0$, we thus will look for $r$ such that with fixed $\gamma $, $\omega''(r)>0$.\\
 We assume that $\mu r\gg 1$ implying that  $\text{sech}(\mu r) \thickapprox 0$ and $\tanh(\mu r) \thickapprox 1$. Therefore we only need to choose $r$ large enough so that 
 
  $$12 \gamma r > \dfrac{(1+3\gamma r^2)^2}{r+\gamma r^3}$$
 
 or 
 
 $$ 3\gamma^2r^4 + 6\gamma r^2 > 1.$$

 \vspace{0.3cm}
 In order to apply Schochet-Weinstein method we will need the condition $\delta \sigma_1 >0$ and  we will only consider the Zakharov-Rubenchik (or Benney-Roskes) system of the form of \eqref{2-D Z-R system 1} satisfying this condition.
 
 \vspace{0.5cm}
 The one-dimensional Zakharov-system possesses solitary wave solutions and in \cite{O}, Serra de Oliveira proved their orbital stability.  One motivation of the present paper was the study of their {\it transverse} stability. The transverse instability of the line solitary wave for some two dimensional models such as the nonlinear Schr\"{o}dinger  equation (NLS), the Kadomtsev-Petviashvili  equation (KP) and some general  "abstract" Hamiltonian systems have been carried out extensively in \cite{ RT, RT1, RT2, Mi, MT}.  
 
 It is thus of interest to study  the transverse stability of the line soliton for the two dimensional model \eqref{2-D Z-R system 1} and the first step is to study the Cauchy problem of  a localized perturbation of \eqref{2-D Z-R system 1} by a line soliton. Another possibility is to consider $y$ or $(y,z)$ - periodic perturbations of the line solitary wave, a first step being to establish the well-posedness of the Cauchy problem for the Zakharov-Rubenchik, Benney-Roskes system in $\R^d\times \T, \; d=1,2,$ which could not result from the methods used in \cite{PoSa} but we achieve here and also in the pure periodic case $\T^{d+1}.$\\

\vspace{0.3cm}
 In order to unify the notation we will rewrite the Benney-Roskes system \eqref{Benney-Roskes model} in the form of \eqref{2-D Z-R system 1}. We replace $(\psi_{01}, \dfrac{\zeta_{10}}{|\textbf{k}|^2 (1-\sigma^2)\sqrt{\mu}}, \dfrac{\psi_{00}}{|\textbf{k}|^2(1-\sigma^2)})$ by $(\psi,\rho,\phi)$ and after calculating the corresponding coefficients, we have:
 \begin{equation}
  \label{notations}
   \left \{
    \begin{aligned}
     & \sigma_3 = - \omega', \, \delta = \dfrac{\epsilon \omega''}{2}, \, \sigma_1 = \dfrac{\epsilon \omega'}{2 |\textbf{k}|}, \, \sigma_2 = \dfrac{2 \epsilon|\textbf{k}|^4 (1-\alpha)}{\omega} \\
     & W = \dfrac{ \epsilon|\textbf{k}|^4 (1-\sigma^2)^2 \sqrt{\mu}}{2 \omega}, \, D= \dfrac{2 \omega}{|\textbf{k}|(1- \sigma^2) \sqrt{\mu}}, \, M = \mu^{-1/4}.
    \end{aligned}
   \right.
 \end{equation}
 
 \vspace{0.3cm}
 The paper is organized as follows. In the next section we reformulate the existence of one-dimensional solitary waves (bright and dark) in our framework. In section 3 we use the Schochet-Weinstein method to prove a local existence for the Benney-Roskes/ Zakharov-Rubenchik system,  keeping the small parameter $\epsilon$  which is relevant for deep water waves. In section 4 we consider the case of a localized perturbation of a line solitary wave. Finally we prove in Section 5 the global existence of   weak solutions perturbing a dark solitary wave.

We conclude the paper by a list of open questions.
 
  %\textcolor{red}{TO BE COMPLETED!}

 \vspace{0.3cm}
\noindent{\bf Notations.} 
%\textcolor{red}{
\begin{itemize}
	\item $\partial_x$ or $(\cdot)_x$ will be used to denote the derivative with respect to variable $x$.
	\item $H^s(D), s\in \R$ denotes the classical Sobolev space in the domain $D.$
	\item $\norma{\cdot}_X$: The norm in a functional space $X$.
	%\item $L^\infty([0,T]; X)$, $L^p_{loc}([0,T];X)$ and $C_w([0,T]; X)$: denote 
	\item $\mathcal{F}$ and $\mathcal{F}^{-1}$ denote the Fourier and inverse Fourier transform respectively.
	\item $\bra{\xi} = \sqrt{1+|\xi|^2}$ for $\xi \in \R^n$ and $\sigma(D)$ denotes the Fourier multiplier with the symbol $\sigma(\xi)$.
	\item $\Re$ and $\Im$ denote the real part and imaginary part of a complex number respectively.
\end{itemize}
%We will denote $|\cdot|_p$ the norm in the Lebesgue space $L^p(\R),\; 1\leq p\leq \infty$ and $\|\cdot\|_s$ the norm in the Sobolev space $H^s(\R^d),\; s\in \R.$ $(\cdot | \cdot)_2$ denotes the scalar product in $L^2.$ We will denote $\hat {f}$ or $\mathcal F(f)$ the Fourier transform of a tempered distribution $f.$ For any $s\in \R,$ we define $|D|^s f$ by its Fourier transform $\widehat{|D|^s f}(\xi)=|\xi|^s \hat{f}(\xi).$ We also denote $|D_x|^s f=\mathcal F^{-1}(|\xi_1|\hat {f})$ and $|D_y|^s f=\mathcal F^{-1}(|\xi_2|\hat {f}).$ Finally we will denote $\Lambda=(I-\Delta)^{1/2}$ and $J_\epsilon=(I-\epsilon \Delta)^{1/2}.$

 %------------------------------------------
 \section{Existence of   one dimensional solitary waves} 
 \label{section existence of 1-d soliton}
 In this section, we reframe the proof of the existence of 1-d solitary waves in \cite{O} in our setting. The 1-d Zakharov-Rubenchik system has the form
 \begin{equation}
 \label{1-D Z-R system 2}
  \left \{
   \begin{aligned}
    & \psi_t - \sigma_3 \psi_x - i \delta \psi_{xx} + i \left \{ \sigma_2 |\psi|^2 + W (\rho + D \phi_x)   \right \} \psi =0 \\
    & \rho_t + \phi_{xx} + D (|\psi|^2)_x = 0\\
    & \phi_t + \dfrac{1}{M^2} \rho + |\psi|^2 =0. 
   \end{aligned}
  \right.
\end{equation}
Setting $\tilde{\phi}(x,t) = \phi_x,$  \eqref{1-D Z-R system 2} becomes
\begin{equation}
 \label{1-D Z-R system 3}
  \left \{
   \begin{aligned}
    & \psi_t - \sigma_3 \psi_x - i \delta \psi_{xx} + i \left \{ \sigma_2 |\psi|^2 + W (\rho + D \tilde{\phi})   \right \} \psi =0 \\
    & \rho_t + \tilde{\phi}_x + D (|\psi|^2)_x =0 \\
    & \tilde{\phi}_t + \dfrac{1}{M^2} \rho_x + (|\psi|^2)_x =0.
   \end{aligned}
  \right.
\end{equation}
Let $c \geq 0$, we look for solutions of the system \eqref{1-D Z-R system 3} of the form 
\[
 ( e^{i \lambda t} K(x-ct), \, a |K(x-ct)|^2, \, b|K(x-ct)|^2).
\]

From the last two equations of \eqref{1-D Z-R system 3} we deduce that 

\begin{equation}\label{param}
a= \dfrac{-(1+cD)}{1/{M^2}-c^2} \quad\text{and}\quad  b= \dfrac{-(c + D/M^2)}{1/M^2-c^2}.
\end{equation}
\\
Then the first equation of \eqref{1-D Z-R system 3} is equivalent to
\begin{equation*}
 \delta \ddot{K} - i (c+ \sigma_3) \dot{K} - \lambda K = \left( \sigma_2 + W (a + b D)  \right) |K|^2 K.
\end{equation*}
Set \[ R(x) = e^{-i(c + \sigma_3)x/2\delta} K(x) \]
then 
\begin{equation}
 \label{elliptic equation soliton}
  \delta \ddot{R} + \left( \dfrac{(c + \sigma_3)^2}{4 \delta} - \lambda   \right) R = \left(  \sigma_2 + W (a + b D)  \right) |R|^2 R.
\end{equation}
The equation \eqref{elliptic equation soliton} has a unique positive solution if :
\begin{equation*}
  \left \{
   \begin{aligned}
    & \dfrac{1}{\delta} \left( \dfrac{(c + \sigma_3)^2}{4 \delta} - \lambda   \right) < 0\\
    & \dfrac{1}{\delta} \left(  \sigma_2 + W (a + b D)  \right) <0
   \end{aligned}
  \right.
\end{equation*}
or equivalently
\begin{equation}
 \label{condition of existence solution}
  \left \{
   \begin{aligned}
    & \dfrac{1}{\delta} \left( \dfrac{(c + \sigma_3)^2}{4 \delta} - \lambda   \right) < 0\\
   &  \dfrac{1}{\delta} \left(  \sigma_2 - \dfrac{W (1 + D^2/M^2 + 2 c D)}{1/M^2 - c^2}  \right) <0
   \end{aligned}
  \right.
\end{equation}   
We see that if $c \rightarrow (1/ M)^-$ and $\lambda$ is large enough then \eqref{condition of existence solution} holds assuming that  $W > 0$ and $\delta > 0$ which  holds true  in both models \eqref{2-D Z-R system 1} and \eqref{Benney-Roskes model} .\\
In this case,
\begin{equation}
\label{bright soliton}
R(x) = \sqrt{ \dfrac{2}{\sigma_2 + W (a+ b D)} \left( \dfrac{(c + \sigma_3)^2}{4 \delta} - \lambda  \right)} \text{sech} (\sqrt{- \dfrac{1}{\delta} \left( \dfrac{(c + \sigma_3)^2}{4 \delta} - \lambda   \right)} \, x)
\end{equation}
Otherwise, if 
\begin{equation*}
 \left \{
  \begin{aligned}
   & \dfrac{1}{\delta} \left( \dfrac{(c + \sigma_3)^2}{4 \delta} - \lambda   \right) > 0\\
   & \dfrac{1}{\delta} \left(  \sigma_2 + W (a + b D)  \right) > 0
  \end{aligned}
 \right.
\end{equation*}
or equivalently
\begin{equation}
 \label{condition of existence of black soliton}
  \left \{
   \begin{aligned}
    & \dfrac{1}{\delta} \left( \dfrac{(c + \sigma_3)^2}{4 \delta} - \lambda   \right) > 0\\
   &  \dfrac{1}{\delta} \left(  \sigma_2 - \dfrac{W (1 + D^2/M^2 + 2 c D)}{1/M^2 - c^2}  \right) > 0
   \end{aligned}
  \right.
\end{equation}   
In the context of  water waves (Benney-Roskes system) there is a regime where the condition \eqref{condition of existence of black soliton} holds. In particular, if we choose $c > \dfrac{1}{ M}$ and $\lambda < \dfrac{(c + \sigma_3)^2}{4 \delta} $ then \eqref{condition of existence of black soliton} holds, since from \eqref{notations} we know that $\delta, \, \sigma_2, \, W $ and $D$ are positive. If $c=0$ then \eqref{condition of existence of black soliton} is equivalent to
\begin{equation*}
 \left \{
  \begin{aligned}
   & \dfrac{\sigma_3^2}{4 \delta} > \lambda \\
   & \dfrac{2 |\textbf{k}|^4 (1- \alpha)}{\omega} - \dfrac{|\textbf{k}|^4 (1- \sigma^2)^2}{2 \omega} - \dfrac{2  \omega|\textbf{k}^2|}{\sqrt{\mu}} >0
  \end{aligned}
 \right.
\end{equation*}
Since $\alpha <0$ one has $\dfrac{2 |\textbf{k}|^4 (1- \alpha)}{\omega} > \dfrac{|\textbf{k}|^4 (1- \sigma^2)^2}{2 \omega} $. Therefore if $\mu$ is large enough (which occurs in the context of deep water waves) then the above conditions hold.

In this case, 
\begin{equation}
\label{dark soliton}
 R(x) = \sqrt{\dfrac{1}{\sigma_2 + W (a+b D)} \left( \dfrac{(c + \sigma_3)^2}{4 \delta} - \lambda \right)} \tanh ( \sqrt{\dfrac{1}{\delta} \left( \dfrac{(c + \sigma_3)^2}{4 \delta} - \lambda \right)} \, x )
\end{equation}

Then the system \eqref{1-D Z-R system 3} has two kind of solitary waves corresponding to the two conditions \eqref{condition of existence solution} and \eqref{condition of existence of black soliton}:
\[
 (  e^{i \lambda t} e^{i (c+ \sigma_3) x / 2 \delta} R(x-ct), \, a \, R^2(x-ct), \, b \, R^2(x-ct)   )
\]
Recalling that $\tilde \phi=\phi_x,$  the solutions of system \eqref{1-D Z-R system 2} should have thus the form 
\begin{equation}
 \label{line soliton}
  Q = (  e^{i \lambda t} e^{i (c+ \sigma_3) x / 2 \delta} R(x-ct), \, a \, R^2(x-ct), \, b \,  P(x-ct)   ).
\end{equation}
Where
\[
 P(x)= \dfrac{\alpha^2}{\beta} \tanh(\beta x)
\]
with
\[
\alpha = \sqrt{ \dfrac{2}{\sigma_2 + W (a+ b D)} \left( \dfrac{(c + \sigma_3)^2}{4 \delta} - \lambda  \right)}, \quad  \beta= \sqrt{- \dfrac{1}{\delta} \left( \dfrac{(c + \sigma_3)^2}{4 \delta} - \lambda   \right)}
\]
in the case $R(x)$ is given by \eqref{bright soliton}.

And 
\[
P(x)= \dfrac{\alpha^2}{\beta}(\beta x-\tanh(\beta x))
\]
with 
\[
 \alpha = \sqrt{ \dfrac{2}{\sigma_2 + W (a+ b D)} \left( \dfrac{(c + \sigma_3)^2}{4 \delta} - \lambda  \right)}, \quad  \beta= \sqrt{ \dfrac{1}{\delta} \left( \dfrac{(c + \sigma_3)^2}{4 \delta} - \lambda   \right)}
\]
in the case $R(x)$ is given by \eqref{dark soliton}.

\textit{Remark}: Similarly  to the case of the cubic nonlinear Schr\"odinger equation, we will call the 1-d solitary wave corresponding to the condition \eqref{condition of existence solution} and \eqref{condition of existence of black soliton} the ``bright'', ``dark'' soliton respectively.
%----------------------------------------
\section{The Z-R/B-R system}
\label{section Z-R/ B-R system}
\label{Section nonperturbed system}
 As aforementioned the asymptotic model \eqref{Benney-Roskes model} is a good approximation of the full water wave system on a time scale $O(1/\epsilon)$ (see \cite{La1} page 233). It is thus crucial to prove the well-posedness of the Cauchy problem on time scales of order $1/\epsilon.$ However, the existence time obtained by using the method in \cite{PoSa} does not reach the $O(1/\epsilon)$ time scale (as we already mentionned it is of order $O(1)$)

  In this section, we give the proof of the  local well-posedness for \eqref{2-D Z-R system 1} by using Schochet-Weinstein method in \cite{Ob}  but  keeping the parameter $\epsilon$ in \eqref{Benney-Roskes model} to estimate the existence time obtained by this method.  It turns out, however, that one does not improve upon the previously known $O(1)$ result (see however the comments in the Introduction).
  
  We consider the following system
\begin{equation}
 \label{2-D Z-R system 1 with epsilon}
  \left \{
   \begin{aligned}
   & \psi_t - \sigma_3 \psi_x - i  \epsilon \delta \psi_{xx} - i \epsilon \sigma_1 \psi_{yy} + i \epsilon \left \{  \sigma_2 |\psi|^2 + W (\rho+ D\phi_x) \right \} \psi = 0 \\
   & \rho_t + \Delta \phi + D (|\psi|^2)_x =0 \\
   & \phi_t + \dfrac{1}{M^2} \rho + |\psi|^2 = 0,
   \end{aligned}
  \right.
\end{equation}

with initial conditions  $(\psi_0,\rho_0,\phi_0)$ for which we obtain a local existence result :
\begin{theorem}
\label{local existence for unperturbed system}
Let $\delta \sigma_1 >0$, $s>2$. Let the initial data $(\psi_0,\rho_0,\phi_0)\in H^{s+1}(\R^2)\times H^s(\R^2)\times H^{s+1}(\R^2)$ and satisfies that
\begin{equation}
\label{V_0}
 WM (-\Delta \phi_0 - \dfrac{D}{M^2} \partial_x \rho_0) = \nabla \cdot V_0,
\end{equation} 
with $V_0 \in H^s(\R^2)^2$.

Then there exist $T>0$ independent of $\epsilon$ and a unique solution $(\psi,\rho,\phi)\in L^\infty(0,T;H^{s+1}(\R^2))\times L^\infty(0,T;H^s(\R^2))\times L^\infty (0,T;H^{s+1}(\R^2)) $ of \eqref{2-D Z-R system 1 with epsilon}.

%Assume that  there exist $V_0(x,y)$ such that
%\begin{equation}
%\color{red}{ WM^2(\rho_t+ D \phi_{xt} + 2D(|\psi|^2)_x)(0,x,y)=\nabla \cdot V_0(x,y)}.
%\end{equation}

\end{theorem}
\begin{remark}\label{3d}
With some minor changes, one obtains the same result in the three-dimensional case, that is $\psi_{yy}$ replaced by $\nabla_\perp \psi.$
\end{remark}
\begin{remark}\label{periodic}

The above theorem and its proof are valid {\it mutatis mutandi} in a periodic $(\T^d), d=2,3,$ or semi-periodic $(\R^{d-1}\times \T)$ setting.

\end{remark}
\begin{proof}

We follow closely the proof in \cite{Ob}, Section 3.3, but we keep track of the parameter $\epsilon$.

We first rewrite \eqref{2-D Z-R system 1 with epsilon} as a dispersive perturbation of a symmetric hyperbolic system.  We take the time derivative of  the second and the third equation of the system \eqref{2-D Z-R system 1}. This allows to decouple the linear parts of those equations
\begin{equation}
 \label{2-D Z-R decoupled equations}
  \left \{
   \begin{aligned}
    & \psi_t - \sigma_3 \psi_x - i \epsilon \delta \psi_{xx} - i \epsilon \sigma_1 \psi_{yy} + i \epsilon \left( \sigma_2 |\psi|^2 + W (\rho + D \phi_x)   \right) \psi =0 \\
    & \rho_{tt} - \Delta \left( \dfrac{1}{M^2} \rho + |\psi|^2  \right) + D (|\psi|^2)_{xt} =0 \\
    & \phi_{tt} - \dfrac{1}{M^2} \left(  \Delta \phi + D (|\psi|^2)_x  \right) + (|\psi|^2)_t = 0.
   \end{aligned}
  \right.
\end{equation}
We set 
\begin{equation} 
\label{Define mathcal U}
  \mathcal{U} =  W \rho +  W D \phi_x.
\end{equation}  

We then get a coupled system for $\psi$ and $\mathcal{U}$
\begin{equation}
 \label{Coupled system}
 \left \{
  \begin{aligned}
   & \psi_t - \sigma_3 \psi_x - i \epsilon \delta \psi_{xx} - i \epsilon \sigma_1\psi_{yy}  + i \epsilon \left(  \sigma_2 |\psi|^2 + \mathcal{U}  \right) \psi =0\\
   & \mathcal{U}_{tt} - \dfrac{1}{M^2} \Delta \mathcal{U} -  W \Delta (|\psi|^2) + 2  \, D W (|\psi|^2)_{xt} - \dfrac{ W D^2}{M^2} (|\psi|^2)_{xx} = 0.
  \end{aligned}
 \right. 
\end{equation}

We can rewrite the second equation in \eqref{Coupled system} as 
\begin{equation}\label{div}
\partial_t(\mathcal U_t+2WD(|\psi|^2)_x)-\nabla\cdot(\frac{1}{M^2}\nabla \mathcal U-\left(\frac{WD^2}{M^2}(|\psi|^2)_x,0)\right)^T=0.
\end{equation}

Integrating with respect to time between $0$ and t and using \eqref{V_0} and the two last equations in \eqref{2-D Z-R system 1 with epsilon}, we infer that $\mathcal U_t+2DW(|\psi|^2)_x$ is a divergence. Then, in order to reduce to a first order system, following the idea in \cite{SW} for the Zakharov system, we define the following  auxiliary (vector valued) function $V$
\begin{equation}
\label{Define V}
 \mathcal U_t+2DW(|\psi|^2)_x=\frac{1}{M}\nabla \cdot V
\end{equation}
%\textcolor{red}{JUSTIFY!}
Plugging this expression into the second equation of \eqref{Coupled system} and assuming that $\mathcal U, V, \psi$ tend to zero at infinity together with their derivatives we infer that 
\[
 \partial_t V = \dfrac{1}{M} \nabla \mathcal{U}+   W M \nabla (|\psi|^2) + \left(  \dfrac{  W D^2}{M} (|\psi|^2)_x,0   \right)^T,
\]
so that we obtain the equivalent first order system
\begin{equation}
 \label{system A U V} 
  \left \{
   \begin{aligned}
    & \psi_t - \sigma_3 \psi_x - i \epsilon \delta \psi_{xx} - i \epsilon \sigma_1 \psi_{yy} + i \epsilon \left(  \sigma_2 |\psi|^2 + \mathcal{U}   \right)\psi  = 0\\
    & \mathcal{U}_t - \dfrac{1}{M} \nabla \cdot V + 2  W D (|\psi|^2)_x = 0 \\
    & \partial_t V - \dfrac{1}{M} \nabla \mathcal{U} -  W M \nabla (|\psi|^2) - \left( \dfrac{  W D^2}{M} (|\psi|^2)_x,  0    \right)^T = 0.
   \end{aligned}
  \right.
\end{equation}
\textit{Remark}: By combining \eqref{V_0}, \eqref{Define mathcal U}, \eqref{Define V}  and the two last equations of \eqref{2-D Z-R system 1 with epsilon}, we have that $V_0$ is actually  the value of  of $V$ at $t=0.$. Recall that  the initial data of $\mathcal{U}$ is defined by $\rho_0$ and $\partial_x \phi_0$.

We now set 
\[
 \mathcal D =\mathcal{U} + \dfrac{  \sigma_1}{2} |\psi|^2
\]
then the system \eqref{system A U V} becomes
\begin{equation}
 \label{system A D V} 
  \left \{
   \begin{aligned}
    & \psi_t - \sigma_3 \psi_x - i \epsilon \delta \psi_{xx} - i \epsilon \sigma_1 \psi_{yy} + i \epsilon \left(   (\sigma_2 - \dfrac{\sigma_1}{2}) |\psi|^2 + \mathcal D  \right) \psi =0 \\
    & \mathcal D_t - \dfrac{1}{M} \nabla . V - \dfrac{ \sigma_1}{2} \left( (|\psi|^2)_t - \sigma_3 (|\psi|^2)_x  \right) + \left(  2  W D - \dfrac{ \sigma_1 \sigma_3}{2}    \right) (|\psi|^2)_x = 0 \\
    & \partial_t V - \dfrac{1}{M} \nabla \mathcal D - \left(  \dfrac{2   W (D^2 + M^2) - \sigma_1}{2M} (|\psi|^2)_x, \, \dfrac{2  W M^2 - \sigma_1}{2M} (|\psi|^2)_y    \right) = 0.
   \end{aligned}
  \right.
\end{equation}
For simplicity of notation we set 
\[
c_1= 2  W D - \dfrac{ \sigma_1 \sigma_3}{2} , \; c_2= \dfrac{2  W (D^2 + M^2) -  \sigma_1}{2M} , \; c_3 =  \dfrac{2  W M^2 -  \sigma_1}{2M}.
\]
Furthermore, we split the $\psi$ in real and imaginary part
\[  \psi = F + i G  \]
and 
\[  \nabla \psi = H + iL = (H_1,H_2)^T + i (L_1, L_2)^T.   \]
Multiplying the first equation of \eqref{system A D V}  by $\bar{\psi}$ and taking the real part, we deduce that 
\begin{equation}
 \label{*}
  \begin{split}
  (|\psi|^2)_t - \sigma_3 (|\psi|)^2_x & = i \epsilon \delta \psi_{xx} \bar{\psi} - i \epsilon \delta \bar{\psi}_{xx} \psi + i \epsilon \sigma_1 \psi_{yy} \bar{\psi} - i \epsilon \sigma_1 \bar{\psi}_{yy} \psi \\
  & = 2 \epsilon \delta (G \partial_x H_1 - F \partial_x L_1) + 2 \epsilon \sigma_1 ( G \partial_y H_2 - F \partial_y L_2).
  \end{split}
\end{equation}
We insert \eqref{*} into \eqref{system A D V} and then separate the real and imaginary parts of the first equation of the system. We furthermore apply the spatial  gradient to the first equation of $\eqref{system A D V}$ to get the equations satisfied by $H$ and $L$. This leads to the following system
\begin{align}
   \label{Main system 1_1}  &  H_t - \sigma_3 H_x + \epsilon \delta L_{xx}  + \epsilon \sigma_1 L_{yy}  -  \epsilon G \nabla \mathcal D \\
    & \qquad - \epsilon \left( \mathcal D L +  \left(  \sigma_2 - \dfrac{\sigma_1}{2}   \right) \left(  (F^2 + G^2) L + 2G (FH + GL)   \right)   \right) = 0 \nonumber \\ 
    \label{Main system 1_2}  & L_t - \sigma_3 L_x - \epsilon \delta H_{xx} - \epsilon \sigma_1 H_{yy} + \epsilon F \nabla \mathcal D \\
    & \qquad + \epsilon \left( \mathcal DH +  \left(  \sigma_2 - \dfrac{\sigma_1}{2}   \right) \left(  (F^2 + G^2) H + 2F (FH + GL)   \right)    \right) =0 \nonumber \\
    & \label{Main system 1_3} F_t - \sigma_3 F_x + \epsilon \delta G_{xx} + \epsilon \sigma_1 G_{yy} - \epsilon \left(   \left(  \sigma_2 - \dfrac{\sigma_1}{2}   \right) (F^2+G^2) + \mathcal D  \right)G = 0\\
    & \label{Main system 1_4} G_t - \sigma_3 G_x - \epsilon \delta F_{xx} - \epsilon \sigma_1 F_{yy} +   \epsilon \left(  \left(  \sigma_2 - \dfrac{\sigma_1}{2}   \right) (F^2+ G^2) + \mathcal D  \right)F =0 \\
    & \label{Main system 1_5}  \mathcal D_t - \dfrac{1}{M} \nabla . V - \epsilon \sigma_1 \delta (G \partial_x H_1 -  F \partial_x L_1) - \epsilon \sigma_1^2(G \partial_y H_2 - F \partial_y L_2)\\
    & \qquad  + 2c_1 (H_1 F + L_1 G) =0  \nonumber \\
    &\label{Main system 1_6} \partial_t V - \dfrac{1}{M} \nabla \mathcal D - 2 \left(  c_2 (H_1 F+ L_1 G),\, c_3(H_2F + L_2G)   \right)^T =0.
\end{align}
Since  $\sigma_1 \delta > 0$ , we can  perform the following  change of variables
\[
 H^*= ( \sqrt{\delta \sigma_1} H_1, \sigma_1 H_2)^T \text{ and } L^* = ( \sqrt{\delta \sigma_1} L_1, \sigma_1 L_2)^T,
\]
and  we then set $U= (H^*, L^*, F, G, \mathcal D, V)^T$. 

Therefore,  \eqref{2-D Z-R system 1 with epsilon} is rewritten as a dispersive (skew adjoint) perturbation of a symmetric hyperbolic system given by 
\begin{equation}
 \label{symmetric form for non perturbed system}
 \begin{split}
&  U_t + \left( \epsilon A_1(U) +B_1 \right) U_x + (\epsilon A_2(U) + B_2) U_y +  C(U) U \\
& \qquad = - K_1 U_{xx}  - K_2 U_{yy},
 \end{split}
\end{equation}
where $A_1, \, A_2, B_1$ and $B_2$ are symmetric matrices, $K_1, \, K_2$ are skew symmetric matrices. 
\[
 A_1(U) = \begin{pmatrix}
0_{3 \times 3} & 0_{3 \times 3} & M_1(U)^T  & 0_{3 \times 2} \\
0_{3 \times 3} & 0_{3 \times 3} & 0_{3 \times 1}  & 0_{3 \times 2} \\
M_1(U) & 0_{1 \times 3} & 0 & 0_{1 \times 2}\\
0_{2 \times 3} & 0_{2 \times 3} & 0_{2 \times 1} & 0_{2 \times 2}
 \end{pmatrix},
\]
with \[
M_1(U) = (-  \sqrt{\delta \sigma_1}G , 0, \sqrt{\delta \sigma_1} F).
\]
\[
A_2(U) = \begin{pmatrix}
0 & 0_{1 \times 3} &0_{1 \times 2} & 0  & 0_{1 \times 2} \\
0_{3 \times 1 } & 0_{3 \times 3} & 0_{3 \times 2} & M_2(U)^T & 0_{3 \times 2} \\
0_{2 \times 1} & 0_{2 \times 3} & 0_{2 \times 2} & 0_{2 \times 1} & 0_{2 \times 2} \\
0 & M_2(U) & 0_{1 \times 2} & 0 & 0_{1 \times 2} \\
0_{2 \times 1} & 0_{2 \times 3} & 0_{2 \times 2} & 0_{2 \times 1} & 0_{2 \times 2}
\end{pmatrix}
\]
with \[
 M_2(U) = (- \sigma_1G,0, \sigma_1 F).
\]

And note that $C(U)$ contains the term that is independent of $\epsilon$ which is
\[
C_1(U) = \begin{pmatrix}
 0 & 0 & 0 & 0 & 0_{1\times 5} \\
 0 & 0 & 0 & 0 & 0_{1\times 5} \\
 0 & 0 & 0 & 0 & 0_{1\times 5} \\
 0 & 0 & 0 & 0 & 0_{1\times 5} \\
 0 & 0 & 0 & 0 & 0_{1\times 5} \\
 0 & 0 & 0 & 0 & 0_{1\times 5} \\
 \dfrac{-2 c_1}{\sqrt{\delta \sigma_1}} F & 0  &  \dfrac{-2 c_1}{\sqrt{\delta \sigma_1}} G & 0 & 0_{1 \times 5} \\
  \dfrac{-2 c_2}{\sqrt{\delta \sigma_1}} F & 0 &  \dfrac{-2 c_2}{\sqrt{\delta \sigma_1}} G & 0  & 0_{1 \times 5}  \\
  0 & \dfrac{-2 c_3}{\sigma_1} F & 0 & \dfrac{-2 c_3}{\sigma_1} G & 0_{1 \times 5}
 \end{pmatrix}
\]

Next,we prove that if the initial data $U(0,x,y)=U_0 \in (H^s(\R^2))^9, s>2$, then there exists $T=T(\norma{U_0}_{(H^s(\R^2))^9})$ such that equation \eqref{symmetric form for non perturbed system} has a unique solution in $L^\infty([0,T], (H^s(\R^2))^9)$.

The  proof of the existence of solution is standard and proceeds via a classical iteration scheme for symmetric hyperbolic system (see \cite{M, KM}). The presence of $C_1(U)$, unfortunately, leads to the existence time of order $0(1)$.

The uniqueness of solution of \eqref{symmetric form for non perturbed system} is classically obtained  by estimating the difference of two solutions since the dispersive part does not contribute  to the $L^2$ energy estimate.

The last step will be the recovering the solution of \eqref{2-D Z-R system 1 with epsilon} from the solution of \eqref{symmetric form for non perturbed system}. 

First, set $\psi = F + i G$, then \eqref{Main system 1_3} and \eqref{Main system 1_4} imply that 
	\[
	\psi_t - \sigma_3 \psi - i \epsilon \delta \psi_{xx} - i \epsilon \sigma_1 \psi_{yy} + i \epsilon \left( (\sigma_2- \dfrac{\sigma_1}{2})|\psi|^2 + \mathcal{D}   \right) \psi=0.
	\]

From \eqref{Main system 1_1}-\eqref{Main system 1_4} we can derive an $L^2$ estimate of $\mathscr{W}= (\nabla F -H, \nabla G - L)$, that implies
	\[
	\norma{\mathscr{W}(t)}^2_{L^2} \leq e^{Ct} \norma{\mathscr{W}(0)}^2_{L^2},
	\]
	thus, $\nabla \psi = H+ iL$.\\
	 Then \eqref{Main system 1_5} -\eqref{Main system 1_6} implies the last two equations of \eqref{Coupled system}.
	
 Next, we are going to recover $(\rho,\phi)$. From \eqref{Coupled system} we already know $\psi$ and $\mathcal U.$
	
	%From the solution $\mathcal{D}$ of the symmetric hyperbolic system we can define the solution $\mathcal{U}$ of  \eqref{Coupled system}. \\
	Let $\phi$ be the unique solution of the following linear wave equation
\[	\phi_{tt} - \dfrac{1}{M^2} \left(  \Delta \phi + D (|\psi|^2)_x  \right) + (|\psi|^2)_t = 0 \]
with initial data given by $\phi(t=0)=\phi_0$ and $\phi_t(t=0)= -\dfrac{1}{M^2} \rho_0 - |\psi_0|^2$.

Define 
\[ \rho = \dfrac{1}{W} (\mathcal{U} - WD \phi_x). \]
Then we get that $(\psi, \rho, \phi)$ solves \eqref{2-D Z-R decoupled equations} uniquely with respect to the given initial data.

Next, let $\tilde{\phi}(t,x,y)$ be the unique solution of the differential equation
\[\tilde{\phi}_t + \dfrac{1}{M^2} \rho + |\psi|^2=0,\]
with initial data $\tilde{\phi}(t=0) = \phi_0$. 

From the second equation in  \eqref{2-D Z-R decoupled equations} we get 

$$\rho_{tt}+\Delta \tilde{\phi}_t+D|\psi|^2_{xt}=0$$
and integrating in time and using the initial data, we get the following system
\begin{equation}
\label{rho tilde phi}
\left \{
\begin{aligned}
& \rho_t + \Delta \tilde{\phi} + D(|\psi|^2)_x=0\\
& \tilde{\phi}_t + \dfrac{1}{M^2} \rho + |\psi|^2 =0.
\end{aligned}
\right.
\end{equation}
Taking the time derivative of the second  equation in \eqref{rho tilde phi} we get 
\[
\tilde{\phi}_{tt} - \dfrac{1}{M^2} (\Delta \tilde{\phi} + D(|\psi|^2)_x) + (|\psi|^2)_t=0.
\]
Note that the initial data is also given by $\tilde{\phi}(t=0)=\phi_0$ and $\tilde{\phi}_t(t=0)= -\dfrac{1}{M^2} \rho_0 - |\psi_0|^2$. Therefore we have $\tilde{\phi} = \phi,$ achieving to prove that  $(\psi,\rho, \phi)$ solves the original Zakharov-Rubenchik system.

%\textcolor{red}{First, recalling that  $\psi = F + i G$, then \eqref{Main system 1_3} and \eqref{Main system 1_4} imply that 
%\[
% \psi - \sigma_3 \psi - i \epsilon \delta \psi_{xx} - i \epsilon \sigma_1 \psi_{yy} + i \epsilon \left( (\sigma_2- \dfrac{\sigma_1}{2})|\psi|^2 + \mathcal{D}   \right) \psi=0.
%\]
%}
%\textcolor{red}{Next, from \eqref{Main system 1_1}-\eqref{Main system 1_4} we can derive an $L^2$ estimate of $\mathscr{W}= (\nabla F -H, \nabla G - L)$, that implies
%\[
% \norma{\mathscr{W}(t)}^2_{L^2} \leq e^{Ct} \norma{\mathscr{W}(0)}^2_{L^2}.
%\]
%Therefore, $\nabla \psi = H+ iL$. And \eqref{Main system 1_5} -\eqref{Main system 1_6} implies the last two equations of \eqref{system A D V}.}

%\textcolor{red}{When $\psi$ is known, in order to find $\rho$ and $\phi$ we only need to solve the two linear wave equations in \eqref{2-D Z-R decoupled equations}  with given initial data.}

\end{proof}
%----------------------------------------------------------
\section{The perturbed Z-R/B-R system}
\label{section Perturbed Z-R/B-R system}
In this section, we consider the Cauchy problem for \eqref{2-D Z-R system 1} when it is perturbed by the line solitary wave $Q$ given by \eqref{line soliton}. That means, we will find solutions of \eqref{2-D Z-R system 1} of the form $(\psi+\phi_1, \rho + \phi_2, \phi + \phi_3)$, where we denote $Q=(\phi_1, \phi_2, \phi_3)$. The new system reads
\begin{equation}
\left \{
 \label{perturbed system}
 \begin{aligned}
  & \psi_t - \sigma_3 \psi_x - i  \epsilon \delta \psi_{xx} - i \epsilon \sigma_1 \psi_{yy} + i \epsilon \left \{  \sigma_2 |\psi|^2 + 2\sigma_2 \Re (\phi_1 \bar{\psi})  + W (\rho+ D\phi_x) \right \} \phi_1  \\
  & \qquad \qquad  + i \epsilon \{ \sigma_2 |\psi+\phi_1|^2 + W(\rho + \phi_2 + D(\phi + \phi_3)_x)\} \psi =0 \\
  & \rho_t + \Delta \phi + D (|\psi|^2 + 2\Re (\phi_1 \bar{\psi}))_x =0 \\
  & \phi_t + \dfrac{1}{M^2} \rho + |\psi|^2 +2\Re (\phi_1 \bar{\psi}) = 0,
 \end{aligned}
 \right.
\end{equation}

\begin{remark} A natural way to solve the Cauchy problem for \eqref{perturbed system} would be to use the "dispersive method" in \cite {PoSa}. However the fact that  the line soliton does not decay to $0$ in the transverse direction leads to a difficulty when dealing with a new nonlocal linear term and seems to preclude to extend this method in a straightforward way. Therefore we will apply the method of the previous section to this case.
\end{remark}
In the  first step, we need to rewrite \eqref{perturbed system} in the form of a skew-adjoint perturbation of a symmetric hyperbolic system. 

Using the fact that $Q=(\phi_1, \phi_2, \phi_3)$ is also a  solution of \eqref{2-D Z-R system 1 with epsilon} with $\epsilon=1$ then by the same  calculations as in the previous  section, we obtain that $(H_r, L_r, F_r, G_r, D_r, V_r)^T$ is a solution of \eqref{Main system 1_1}-\eqref{Main system 1_6} with $\epsilon=1$, where
\begin{equation}
 \label{residual}
 \left \{
  \begin{aligned}
   & H_r = \nabla ( \Re \phi_1), \quad L_r = \nabla (\Im \phi_1), \\
   & F_r = \Re \phi_1, \quad G_r = \Im \phi_1, \\
   &\mathcal  D_r = \mathcal U_r + \dfrac{\sigma_1}{2} |\phi_1|^2 \text{ with  } \mathcal U_r = W \phi_2 + W D (\phi_3)_x ,\\
   & \partial_t V_r - \dfrac{1}{M} \nabla \mathcal U_r - W M \nabla (|\phi_1|^2) - \left( \dfrac{W D^2}{M} (|\phi_1|^2)_x, 0  \right)^T = 0.
  \end{aligned}
 \right.
\end{equation}
Similarly, if $(\psi,\rho, \phi)$ is a solution of \eqref{perturbed system} or if $(\psi+ \phi_1, \rho+ \phi_2, \phi+ \phi_3)$ is a solution of \eqref{2-D Z-R system 1 with epsilon} (with $\epsilon=1$), then $(\tilde{H}, \tilde{L}, \tilde{F}, \tilde{G}, \tilde{\mathcal D}, \tilde{V})^T$ is a solution of \eqref{Main system 1_1}-\eqref{Main system 1_6} with $\epsilon=1$, where
\begin{equation}
 \label{residual 2}
  \left \{
   \begin{aligned}
    & \tilde{H} = \nabla ( \Re (\psi+ \phi_1)), \quad \tilde{L} = \nabla ( \Im (\psi+ \phi_1)), \\
    & \tilde{F} = \Re (\psi+\phi_1) , \quad \tilde{G} = \Im (\psi + \phi_1),\\
    & \tilde{\mathcal D} =\tilde{ \mathcal U} + \dfrac{\sigma_1}{2} |\psi + \phi_1|^2 \text{ with } \tilde{\mathcal U} = W (\rho + \phi_2) + W D (\phi+ \phi_3)_x , \\
    & \partial_t \tilde{V} - \dfrac{1}{M} \nabla \tilde{\mathcal U} - W M \nabla (|\psi+ \phi_1|^2) - \left(  \dfrac{W D^2}{M} (|\psi+ \phi_1|^2)_x, 0  \right)^T =0.
   \end{aligned}
  \right.
\end{equation}
We now set $(H,L,F,G,\mathcal{D},V)^T = (\tilde{H}, \tilde{L}, \tilde{F}, \tilde{G}, \tilde{\mathcal D}, \tilde{V})^T - (H_r, L_r, F_r, G_r, D_r, V_r)^T$, more precisely, we have
\begin{equation}
 \label{residual 3} 
  \left \{
   \begin{aligned}
    & H = \nabla (\Re \psi), \quad L= \nabla (\Im \psi) , \\
    & F = \Re \psi, \quad G = \Im \psi , \\
    & \mathcal D = \mathcal U + \dfrac{\sigma_1}{2} (|\psi|^2 + 2 \Re (\phi_1 \bar{\psi} ))  \text{ with } \mathcal U = W \rho + W D \phi_x , \\
    & \partial_t V - \dfrac{1}{M} \nabla \mathcal U - W M \nabla (|\psi|^2 + 2 \Re(\phi_1 \bar{\psi})) - \left(  \dfrac{W D^2}{M} (|\psi|^2 + 2 \Re (\phi_1 \bar{\psi}))_x, 0 \right)^T =0.
   \end{aligned}
  \right.
\end{equation}
  
\textit{Remark}: In order to define the initial data for $V$, we use the following expression
 \begin{align*}
 \dfrac{1}{M} \nabla \cdot V & = \partial_t(\tilde{ \mathcal U} - \tilde{ \mathcal U}_r) + 2 DW (|\psi|^2 + 2 \Re (\phi_1 \bar{\psi}))_x \\
 & = -W \Delta \phi -\dfrac{DW}{M^2} \rho_x.
 \end{align*}
Combining \eqref{Main system 1_1}-\eqref{Main system 1_6}, \eqref{residual} and \eqref{residual 2}, it transpires that  $(H,L,F,G,\mathcal D,V)^T$ is a solution of 
\begin{equation}
 \label{main system 2}
  \left \{
   \begin{aligned}
   & H_t - \sigma_3 H_x + \delta L_{xx} + \sigma_1 L_{yy} - (G + G_r) \nabla \mathcal D  + \mathcal{R}_1 =0 \\
   & L_t - \sigma_3 L_x - \delta H_{xx} - \sigma_1 H_{yy} + (F + F_r) \nabla \mathcal D + \mathcal{R}_2 =0 \\
   & F_t - \sigma_3 F_x + \delta G_{xx} + \sigma_1 G_{yy} + \mathcal{R}_3 =0 \\
   & G_t - \sigma_3 G_x - \delta F_{xx} - \sigma_1 F_{yy} + \mathcal{R}_4 = 0\\
   & \mathcal D_t - \dfrac{1}{M} \nabla \cdot V -\sigma_1 \delta \left( (G+G_r)\partial_x H_1 - (F +F_r) \partial_x L_1  \right) \\
   & \qquad - \sigma_1^2 \left( (G+ G_r) \partial_y H_2 - (F+F_r) \partial_y L_2  \right) + \mathcal{R}_5 =0 \\
   & \partial_t V - \dfrac{1}{M} \nabla \mathcal D + \mathcal{R}_6 =0.
   \end{aligned}
  \right.
\end{equation}
Where
\[
 \begin{aligned}
 & \mathcal{R}_1 \\
 & =- G \nabla \mathcal D_r - \tilde{ \mathcal D}L - \mathcal D L_r \\
 & \quad - \left( \sigma_2 - \dfrac{\sigma_1}{2}  \right) \Bigg ( \left( \tilde{F}^2 + \tilde{G}^2   \right) L  + \left( F^2 + G^2 + 2 F F_r + 2G G_r  \right) L_r  \\
 & \qquad + 2 G \left( \tilde{H} \tilde{F} + \tilde{G} \tilde{L}  \right) + 2 G_r \left( \tilde{H} F + H F_r + \tilde{L} G + L G_r  \right) \Bigg ),
 \end{aligned}
\]

\[
 \begin{aligned}
  & \mathcal{R}_2 \\
  &  =F \nabla \mathcal D_r + \tilde{\mathcal D} H + \mathcal D H_r  \\
  & \quad+ \left( \sigma_2 - \dfrac{\sigma_1}{2}  \right) \Bigg ( (\tilde{F}^2 + \tilde{G}^2) H + \left( F^2 + G^2 + 2F F_r + 2G G_r  \right) H_r \\
  & \quad \qquad + 2 F \left( \tilde{H} \tilde{F} + \tilde{G} \tilde{L}  \right) + 2 F_r \left( \tilde{H} F + H F_r + \tilde{G} L + G L_r  \right) \Bigg ),
 \end{aligned}
\]

\[
\begin{aligned}
 & \mathcal{R}_3  = - \Bigg ( \left( \sigma_2 - \dfrac{\sigma_1}{2}  \right) \left( \tilde{F}^2 + \tilde{G}^2  \right) + \tilde{\mathcal D}  \Bigg ) G \\
 & \quad \qquad  - \Bigg ( \left( \sigma_2 - \dfrac{\sigma_1}{2}  \right) \left( F^2 + G^2  + 2 F F_r + 2 G G_r \right) + \mathcal D \Bigg ) G_r,
\end{aligned}
\]

\[
\begin{aligned}
 & \mathcal{R}_4 =  \Bigg ( \left( \sigma_2 - \dfrac{\sigma_1}{2}  \right) \left( \tilde{F}^2 + \tilde{G}^2  \right) + \tilde{\mathcal D}  \Bigg ) F \\
 & \quad \qquad  + \Bigg ( \left( \sigma_2 - \dfrac{\sigma_1}{2}  \right) \left( F^2 + G^2  + 2 F F_r + 2 G G_r \right) + \mathcal D \Bigg ) F_r,
\end{aligned}
\]

\[
 \begin{aligned}
  & \mathcal{R}_5 = \bigg ( \left(  2c_1 \tilde{H_1} + \sigma_1 \delta \Im (\phi_1)_{xx}  \right) F + \left(  2c_1 \tilde{L} - \sigma_1 \delta \Re (\phi_1)_{xx}  \right)G \bigg )\\
  & \quad \qquad + 2c_1( H_1 F_r + L_1 G_r ),
 \end{aligned}
\]

\[
 \begin{aligned}
  & \mathcal{R}_6 \\
  & \quad  = -2 \bigg ( c_2 (\tilde{H_1} F + F_r H_1 + \tilde{L_1} G + G_r L_1); c_3 ( \tilde{H_2} F + \tilde{L_2} G + F_r H_2 + G_r L_2  ) \bigg )^T.
 \end{aligned}
\]
Similarly to the last Section, if $\rho_1 \delta > 0$ , we can  change variables as follows
\[
 H^*= ( \sqrt{\delta \sigma_1} H_1, \sigma_1 H_2)^T \text{ and } L^* = (\sqrt{\delta \sigma_1} L_1, \sigma_1 L_2)^T,
\]
and then we set $U= (H^*, L^*, F, G, \mathcal D, V)^T$. Therefore,  the perturbation of \eqref{2-D Z-R system 1} by the line solitary wave $Q$ is rewritten as a dispersive perturbation of a symmetric hyperbolic given by 
\begin{equation}
 \label{symmetric form for  perturbed system}
 \begin{split}
&  U_t + \left(  A_1(U) +B_1 (\phi_1) + C_1 \right) U_x + (A_2(U) + B_2(\phi_1) + C_2) U_y \\
& \qquad + (C(U)+ \tilde{C}(Q)) U  = - K_1 U_{xx}  - K_2 U_{yy}.
 \end{split}
\end{equation}
Where, with $j\in \{ 1,2\}$,  $A_j, \, B_j, \, C_j $ are symmetric matrices, $K_j$ are skew symmetric and $C_j$ are constant matrices. $A_j$ have the same form as in the proof of Theorem \ref{local existence for unperturbed system}, $C(U)$ contains quadratic and  linear elements, and $B_j$ have the form
\[
 B_1(\phi_1) = \begin{pmatrix}
0_{3 \times 3} & 0_{3 \times 3} & N_1(\phi_1)^T  & 0_{3 \times 2} \\
0_{3 \times 3} & 0_{3 \times 3} & 0_{3 \times 1}  & 0_{3 \times 2} \\
N_1(\phi_1) & 0_{1 \times 3} & 0 & 0_{1 \times 2}\\
0_{2 \times 3} & 0_{2 \times 3} & 0_{2 \times 1} & 0_{2 \times 2}
 \end{pmatrix},
\]
with \[
N_1(\phi_1) = (-  \sqrt{\delta \sigma_1}G_r , 0, \sqrt{\delta \sigma_1} F_r).
\]
\[
B_2(\phi_1) = \begin{pmatrix}
0 & 0_{1 \times 3} &0_{1 \times 2} & 0  & 0_{1 \times 2} \\
0_{3 \times 1 } & 0_{3 \times 3} & 0_{3 \times 2} & N_2(\phi_1)^T & 0_{3 \times 2} \\
0_{2 \times 1} & 0_{2 \times 3} & 0_{2 \times 2} & 0_{2 \times 1} & 0_{2 \times 2} \\
0 & N_2(\phi_1) & 0_{1 \times 2} & 0 & 0_{1 \times 2} \\
0_{2 \times 1} & 0_{2 \times 3} & 0_{2 \times 2} & 0_{2 \times 1} & 0_{2 \times 2}

\end{pmatrix}
\]
with \[
 N_2(\phi_1) = (-\sigma_1 G_r,0, \sigma_1 F_r).
\]
Furthermore, note that the matrix $\tilde{C}(Q)$ depends on $(\phi_1, \phi_2, \partial_x \phi_3)$ making the following analysis holds for both cases when $Q$ is the bright or the dark soliton.\\
We have written the perturbation of \eqref{2-D Z-R system 1} by the line solitary wave $Q$ to the form of a symmetric hyperbolic system. Applying the same method as in the proof of Theorem \ref{local existence for unperturbed system} we obtain the following result.

\begin{theorem}
 \label{local existence for perturbed system}
Let $\delta \sigma_1 >0$, $s>2$. Let the initial data $(\psi_0,\rho_0,\phi_0)\in H^{s+1}(\R^2)\times H^s(\R^2)\times H^{s+1}(\R^2)$ and satisfies that
\begin{equation}
\label{V_0 2}
WM (-\Delta \phi_0 - \dfrac{D}{M^2} \partial_x \rho_0) = \nabla \cdot V_0,
\end{equation} 
with $V_0 \in H^s(\R^2)^2$.

Then there exists $T>0$ and a unique solution $(\psi,\rho,\phi)\in L^\infty([0,T];H^{s+1}(\R^2))\times  L^\infty([0,T];H^{s}(\R^2))\times L^\infty([0,T];H^{s+1}(\R^2))$ of  \eqref{2-D Z-R system 1} when it is perturbed by the line soliton $Q=(\phi_1, \phi_2, \phi_3)$.
\end{theorem}
\begin{proof}
 The proof of Theorem \ref{local existence for perturbed system} and Theorem \ref{local existence for unperturbed system} are essentially the same except the estimates for the terms $B_j (\phi_1)$ and $\tilde{C}( Q)$.\\
In the proof of Theorem \ref{local existence for unperturbed system}, in order to estimate the derivative of order $s$, we use the commutator estimate and the Bessel potential $J^s = (1-\Delta)^{s/2}$. Although, in this case, the 1-D soliton solution $Q$ does not decay in ``$y$'' direction that make that argument is not true. Therefore, we use the fact that $J^s \thicksim J^s_x + J^s_y$, where $J_x = \mathcal{F}^{-1}\bra{\xi_1} \mathcal{F}, \, J_y=\mathcal{F}^{-1}\bra{\xi_2}\mathcal{F} $. Hence, we only need to estimate $J^s_x U$ and $J^s_y U$ instead of $J^sU$. Since $Q$ is independent of $y$, $J^s_y$ is harmless and since $\norma{J^s_x U}_{L^2} = \norma{\norma{J^s_x U}_{L^2_x}}_{L^2_y}$, we can apply the commutator estimate in one dimensional case.

 The rest of the proof proceeds exactly  as in the proof of Theorem \ref{local existence for unperturbed system}.
Again we emphasize that the same result holds true mutatis mutandi in a $\R‚\times \T$ setting, a framework that would be needed to study the stability of the line soliton with respect to periodic transverse perturbations.
\end{proof}
\vspace{0.5cm}

\section{Global solution}
\label{section Global solution}
In this section we will establish the conservation of energy for the perturbation of \eqref{2-D Z-R system 1} by the line soliton $Q$ given in Section \ref{section existence of 1-d soliton}  and as a consequence, the existence of a global weak solution when $Q$ is the dark soliton.

In order to make the calculation easier, we will consider the solution of the form 
\[ (e^{i\lambda t} e^{ i \dfrac{\sigma_3}{2 \delta} x} \psi(x,y,t), \rho(x,y,t), \phi(x,y,t)) \]
then the 1-d solitary wave will have the following form
\[  Q = (\phi_1, \phi_2, \phi_3) = (R(x), a R^2(x), b P(x)),   \]
with $R(x)$, $P(x)$ are given in Section \ref{section existence of 1-d soliton} (note that this trick will not affect  the analysis in section \ref{section Perturbed Z-R/B-R system}). \\
Then the system \eqref{perturbed system} becomes
\begin{equation}
 \label{perturbed 2-D Z-R system 1}
  \left \{
   \begin{aligned}
   & \psi_t + i (\lambda - \dfrac{\sigma_3^2}{4 \delta} ) \psi  - i \delta \psi_{xx} - i \sigma_1 \psi_{yy} \\
   & \quad + i \left \{  \sigma_2 |\psi|^2 + W (\rho+ D\rho_x) + 2 \sigma_2 \phi_1 Re (  \psi) + |\phi_1|^2 + W (\phi_2 + D \partial_x \phi_3) \right \} \psi \\
   & \quad + i \{  \sigma_2 |\psi|^2 + W (\rho+ D\phi_x) + 2 \sigma_2 \phi_1  Re  ( \psi)  \} \phi_1 = 0 \\
   & \rho_t + \Delta \phi + D (|\psi|^2 + 2 \phi_1  \Re (\psi))_x =0 \\
   & \phi_t + \dfrac{1}{M^2} \rho + |\psi|^2 + 2 \phi_1  \Re (\psi) = 0,
   \end{aligned}
  \right.
\end{equation}
In this section we will establish the energy conservation for \eqref{perturbed 2-D Z-R system 1} when it is perturbed by a line soliton $Q$ and the existence of a global weak solution when $Q$ is the dark soliton.
\begin{theorem}
 \label{Th_energy conservation}
  Let $(\psi,\rho,\phi)$ be a solution of system \eqref{perturbed 2-D Z-R system 1} obtained in Theorem \ref{local existence for perturbed system}, defined in the time interval $[0,T]$. Then the quantity
  \begin{equation}
   \label{energy conservation}
   \begin{aligned}
   E  & = (\lambda -\dfrac{\sigma_3^2}{4\delta}) \norma{\psi}_{L^2}^2  + \delta \norma{\psi_x}_{L^2}^2 + \sigma_1 \norma{\psi_y}_{L^2}^2 \\
   & \quad + \dfrac{\sigma_2}{2} \norma{|\psi|^2 + 2 \phi_1 \Re(\psi)}_{L^2}^2 + \sigma_2 \norma{\phi_1 \psi}_{L^2}^2 + \dfrac{W}{2 M^2} \norma{\rho}_{L^2}^2 + \dfrac{W}{2} \norma{\nabla \phi}_{L^2}^2 \\
   & \quad - \int W (M^2 + D^2) |\phi_1|^2 |\psi|^2 + \int W (\rho + D \phi_x) (|\psi|^2 + 2 \phi_1 \Re (\psi)).
   \end{aligned} 
  \end{equation}
  is conserved for $t \in [0,T]$.
\end{theorem}
\begin{proof}
 We multiply the first equation in \eqref{perturbed 2-D Z-R system 1} by $\partial_t \bar{\psi}$, integrate the result and take its imaginary part to get successively
 \begin{equation}
 \label{energy 1}
  (\lambda - \dfrac{\sigma_3^2}{4 \delta} ) \Re \int \psi \bar{\psi}_t = \dfrac{1}{2} (\lambda - \dfrac{\sigma_3^2}{4 \delta} ) \int (|\psi|^2)_t,
 \end{equation}
% \item
% \[
 % -\rho_3 Im \int A_x \bar{A}_t = \rho_3 Im \int A (\bar{A}_x)_t = \dfrac{1}{2} %\sigma_3 Im \int \bar{A}_x A_t + A (\bar{A}_x)_t = \dfrac{1}{2} \rho_3 Im \int \partial_t (\bar{A}_x A)
% \]
 \begin{equation}
 \label{energy 2}
  - \delta \Re \int \psi_{xx} \bar{\psi}_t = \delta \Re \int \psi_x \bar{\psi}_{xt} = \dfrac{1}{2} \delta \int (|\psi_x|^2)_t,
 \end{equation}
 \begin{equation}
 \label{energy 3}
  - \sigma_1 \Re \int \psi_{yy} \bar{\psi}_t = \dfrac{1}{2} \sigma_1 \int (|\psi_y|^2)_t,
 \end{equation}
 \begin{equation}
 \label{energy 4}
  \begin{aligned}
   & \Re \int \left(  \sigma_2 |\psi|^2 + W (\rho+ D\phi_x) + 2 \sigma_2 \phi_1 \Re ( \psi) + \sigma_2 |\phi_1|^2 + W (\phi_2 + D \partial_x \phi_3)   \right) \psi \bar{\psi}_t \\
   & \quad = \dfrac{1}{2} \int  \left(  \sigma_2 |\psi|^2 + W (\rho+ D\phi_x) + 2 \sigma_2  \phi_1 \Re (\psi) + \sigma_2 |\phi_1|^2 + W (\phi_2 + D \partial_x \phi_3)   \right) (|\psi|^2)_t \\
   & \quad = \dfrac{1}{2} \int \dfrac{1}{2} \sigma_2 (|\psi|^4)_t + W (\rho+D \phi_x) (|\psi|^2)_t + 2 \sigma_2 \phi_1  \Re (\psi) (|\psi|^2)_t  \\
   & \qquad + \left( \left( \sigma_2 |\phi_1|^2 + W (\phi_2 + D \partial_x \phi_3)  \right) |\psi|^2 \right)_t,
  \end{aligned}
 \end{equation}
 \begin{equation}
 \label{energy 5}
  \begin{aligned}
   & \Re \int \left(  \sigma_2 |\psi|^2 + W (\rho+ D\phi_x) + 2 \sigma_2  \phi_1  \Re (\psi)  \right) \phi_1 \bar{\psi}_t \\
   & \quad = \int \sigma_2 |\psi|^2 (\phi_1  \Re (\psi))_t + W (\rho+ D \phi_x) (\phi_1 \Re(\psi))_t + \sigma_2 | \phi_1  \Re(\psi)|^2_t.
  \end{aligned}
  \end{equation}
 Combining \eqref{energy 1}, \eqref{energy 2}, \eqref{energy 3}, \eqref{energy 4} and \eqref{energy 5} we obtain
 \begin{align*}
  0 & = \dfrac{d}{dt} \int \bigg ( \dfrac{1}{2} (\lambda - \dfrac{\sigma_3^2}{4 \delta}) |\psi|^2 + \dfrac{\delta}{2} |\psi_x|^2 + \dfrac{\sigma_1}{2} |\psi_y|^2 + \dfrac{\sigma_2}{4} |\psi|^4 \\
  & \quad + \sigma_2 |\psi|^2 \phi_1  \Re (\psi) + \dfrac{1}{2} \left( \sigma_2 |\phi_1|^2 + W (\phi_2 + D \partial_x \phi_3)  \right) |\psi|^2 + \sigma_2 |\phi_1 \Re (\psi)|^2 \\
  & \quad + \dfrac{1}{2} W (\rho + D \phi_x) (|\psi|^2 + 2 \phi_1  \Re(\psi) ) \bigg ) \\
  & \quad - \dfrac{1}{2} \int W (\rho_t + D  \phi_{xt})(|\psi|^2 + 2 \phi_1 \Re( \psi) )
 \end{align*}
 From the second and the third equation in \eqref{perturbed 2-D Z-R system 1}, we get
 \[
  \int \rho_t (|\psi|^2 + 2 \phi_1 \Re(\psi) ) = - \int \rho_t (\phi_t + \dfrac{1}{M^2} \rho ) = - \int \rho_t \phi_t - \dfrac{1}{2 M^2} \int \rho^2_t ,
 \]
 and 
 \begin{align*}
  \int D  \phi_{xt} (|\psi|^2 + 2 \phi_1 \Re(\psi) ) & = - \int \phi_t  \, D (|\psi|^2 + 2 \phi_1  \Re(\psi) )_x \\
  & = \int \phi_t (\rho_t + \Delta \phi ) \\
  & = \int \phi_t \rho_t -\dfrac{1}{2} \int (|\nabla \phi|^2)_t.
 \end{align*}
 That implies
 \begin{align*}
  0 & = \dfrac{d}{dt} \int \bigg ( \dfrac{1}{2} (\lambda -\dfrac{\sigma_3^2}{4 \delta}) |\psi|^2  + \dfrac{\delta}{2} |\psi_x|^2 + \dfrac{\sigma_1}{2} |\psi_y|^2 + \dfrac{\rho_2}{4} |\psi|^4 \\
  & \quad + \sigma_2 |\psi|^2 \phi_1 \Re (\psi) + \dfrac{1}{2} \left( \sigma_2 |\phi_1|^2 + W (\phi_2 + D \partial_x \phi_3)  \right) |\psi|^2 + \sigma_2 |\phi_1 \Re (\psi)|^2 \\
  & \quad + \dfrac{1}{2} W (\rho + D \phi_x) (|\psi|^2 + 2 \phi_1 \Re(\psi) )  \\
  & \quad + \dfrac{W}{4 M^2} \rho^2   + \dfrac{W}{4} |\nabla \phi|^2 \bigg )
 \end{align*}
 Finally we get the energy conservation
 \begin{equation*}
  \begin{aligned}
   E  & = (\lambda -\dfrac{\sigma_3^2}{4\delta}) \norma{\psi}_{L^2}^2  + \delta \norma{\psi_x}_{L^2}^2 + \sigma_1 \norma{\psi_y}_{L^2}^2 + \dfrac{\sigma_2}{2} \norma{\psi}_{L^4}^4 \\
   & \quad + 2 \sigma_2 \int |\psi|^2 \phi_1  \Re (\psi) + \sigma_2 \norma{\phi_1 \psi}_{L^2}^2  + 2 \sigma_2 \norma{ \phi_1 \Re(\psi)}_{L^2}^2\\
   & \quad +  \int W (\phi_2 + D \partial_x \phi_3) |\psi|^2 + \dfrac{W}{2 M^2} \norma{\rho}_{L^2}^2 + \dfrac{W}{2} \norma{\nabla \phi}_{L^2}^2 \\
   & \quad + \int W (\rho + D \phi_x) (|\psi|^2 + 2 \phi_1 \Re (\psi)) \\
   & = (\lambda -\dfrac{\sigma_3^2}{4\delta}) \norma{\psi}_{L^2}^2  + \delta \norma{\psi_x}_{L^2}^2 + \sigma_1 \norma{\psi_y}_{L^2}^2 \\
   & \quad + \dfrac{\sigma_2}{2} \norma{|\psi|^2 + 2 \phi_1 \Re(\psi)}_{L^2}^2 + \sigma_2 \norma{\phi_1 \psi}_{L^2}^2 + \dfrac{W}{2 M^2} \norma{\rho}_{L^2}^2 + \dfrac{W}{2} \norma{\nabla \phi}_{L^2}^2 \\
   & \quad - \int W (M^2 + D^2) |\phi_1|^2 |\psi|^2 + \int W (\rho + D \phi_x) (|\psi|^2 + 2 \phi_1 \Re (\psi)).
   \end{aligned}
 \end{equation*}

Note that we got rid of the terms involving $\phi_2$ and $\phi_3$ since by \eqref{param} one has $a+Db=-(M^2+D^2)$ implying that
$\phi_2+D\partial_x\phi_3=-(M^2+D^2)\phi_1^2.$
\end{proof}
\begin{theorem} Assume that $Q$ is the dark soliton given by \eqref{dark soliton} with  wave speed $c=0$.\\
i) Let $(\psi, \rho, \phi)$ be the solution of \eqref{perturbed 2-D Z-R system 1} obtained by Theorem \ref{local existence for perturbed system} with existence time interval $[0,T]$. Then for all $t \in [0,T]$ we have
 \begin{equation}
  \label{energy bound}
   \norma{\psi(t)}_{H^1} + \norma{\rho(t)}_{L^2}  + \norma{\phi(t)}_{H^1} \leq C(t).
 \end{equation}
ii) For any  $(\psi_0, \rho_0, \phi_0) \in H^1 \times L^2 \times H^1$, there exists a global weak solution $(\psi, \rho, \phi)$ of \eqref{perturbed 2-D Z-R system 1} such that for any $T>0$
\begin{equation}
 \label{global solution}
 \begin{aligned}
 & \psi, \phi \in L^\infty ([0,T]: H^1), \; \rho \in L^\infty([0,T]: L^2) \\
 & \psi_t, \rho_t \in L^\infty([0,T]: H^{-1}), \; \phi_t \in L^\infty ([0,T]: L^2).
 \end{aligned}
\end{equation}
\end{theorem}
\begin{proof}
i) For any $ \varepsilon \in (0,1)$, using Cauchy inequality we have
\[
 \left | \int W \rho (|\psi|^2 + 2 \phi_1 \Re (\psi)) \right | \leq \dfrac{W}{2 M^2} (1- \varepsilon) \int \rho^2 + \dfrac{M^2 W}{2 (1- \varepsilon)} \int (|\psi|^2 + 2 \phi_1 \Re (\psi))^2 
\]
and
\[
 \left | \int W D \phi_x (|\psi|^2 + 2 \phi_1 \Re (\psi)) \right | \leq \dfrac{W}{2} (1- \varepsilon) \int |\phi_x|^2 + \dfrac{W D^2}{2(1-\varepsilon)} \int (|\psi|^2 + 2 \phi_1 \Re (\psi))^2,
\]
then 
\begin{align*}
&\left | \int W (\rho + D \phi_x) (|\psi|^2 + 2 \phi_1 \Re (\psi)) \right | \\
& \quad \leq \dfrac{W (M^2 + D^2)}{2 (1- \varepsilon)} \norma{|\psi|^2 + 2 \phi_1 \Re(\psi)}_{L^2}^2 + \dfrac{W}{2 M^2} (1- \varepsilon) \norma{\rho}_{L^2}^2 \\
& \qquad + \dfrac{W}{2} (1- \varepsilon) \norma{\nabla\phi}_{L^2}^2.
\end{align*}
Note that we are considering the stationary dark soliton, so that  the condition \eqref{condition of existence of black soliton} becomes
\begin{equation*}
  \left \{
   \begin{aligned}
    & \dfrac{1}{\delta} \left( \dfrac{\sigma_3^2}{4 \delta} - \lambda   \right) > 0\\
   &  \dfrac{1}{\delta} \left(  \sigma_2 - W (M^2 + D^2) \right) > 0,
   \end{aligned}
  \right.
\end{equation*} 

so there exists $\varepsilon >0$ small enough such that 
\[
 \sigma_2 > \dfrac{W (M^2 + D^2)}{(1- \varepsilon)}.
\]
Therefore, the conservation law \eqref{energy conservation} implies
\begin{equation}
\label{L2 bound 1}
\begin{aligned}
& \dfrac{1}{2} (\sigma_2 - \dfrac{W (M^2 + D^2)}{(1- \varepsilon)} ) \norma{|\psi|^2 + 2 \phi_1 \Re(\psi)}_{L^2}^2 + \varepsilon \dfrac{W}{2 M^2} \norma{\rho}_{L^2}^2 + \varepsilon \dfrac{W}{2} \norma{\nabla \phi}_{L^2}^2 \\
& \quad \leq E + (\dfrac{\sigma_3^2}{4\delta} - \lambda) \norma{\psi}_{L^2}^2.
\end{aligned}
\end{equation}
From now we will fix such an $\varepsilon$ and define
\[
 c_1= \dfrac{1}{2} (\sigma_2 - \dfrac{W (M + D^2)}{(1- \varepsilon)} ), \, c_2 = \varepsilon \dfrac{W}{2 M} , \, c_3 = \varepsilon \dfrac{W}{2}.
\]
The first equation of \eqref{perturbed 2-D Z-R system 1}  implies
\begin{align*}
\dfrac{1}{2} \dfrac{d}{dt} \norma{\psi}_{L^2}^2 & = \Im \int ( \sigma_2 ( |\psi|^2 + 2 \phi_1 \Re(\psi) ) + W(\rho + D \phi_x) ) \phi_1 \bar{\psi} \\
& \leq c_1 \norma{|\psi|^2 + 2 \phi_1 \Re(\psi)}_{L^2}^2 + c_2 \norma{\rho}_{L^2}^2 + c_3 \norma{\phi_x}_{L^2}^2 \\
& \qquad + ( \dfrac{\sigma_2^2}{4 c_1} + \dfrac{W^2}{4 c_2} + \dfrac{W^2 D^2}{4 c_3} ) \norma{\phi_1 \Im(\psi)}_{L^2}^2.
\end{align*}
Combining with \eqref{L2 bound 1} we get
\begin{equation*}
 \begin{aligned}
  \dfrac{1}{2} \dfrac{d}{dt} \norma{\psi}_{L^2}^2 & \leq E + (\dfrac{\sigma_3^2}{4\delta} - \lambda) \norma{\psi}_{L^2}^2  \\
  & \quad +  ( \dfrac{\sigma_2^2}{4 c_1} + \dfrac{W^2}{4 c_2} + \dfrac{W^2 D^2}{4 c_3} ) \norma{\phi_1 \Im(\psi)}_{L^2}^2.
 \end{aligned}
\end{equation*}
Since $\phi_1 \in L^\infty,$  $\norma{\phi_1 \Im (\psi)}_{L^2}$ is under control and we recall that $\dfrac{\sigma_3^2}{4 \delta} - \lambda >0$. Hence, by Gronwall inequality we can bound  $\norma{\psi(t)}_{L^2}$ by a constant depending on $t$. The bounds on $\norma{\nabla \psi}_{L^2}, \, \norma{\rho}_{L^2}, \, \norma{\nabla \phi}_{L^2}$ then follow from the energy conservation.\\

\vspace{0.3cm}	

ii)  We shall use a classical compactness method (see for instance \cite{Li}). The estimates in (i) prove that regular solutions of \eqref{perturbed 2-D Z-R system 1} are uniformly bounded in $H^1\times L^2\times H^1$. To obtain global weak solutions, as in \cite{SuSu2}, we  first implement a   Galerkin approximation process (possibly after smoothing the initial data) , yielding  a sequence of approximate solutions $(\psi_m,\rho_m,\phi_m)$ of \eqref{perturbed 2-D Z-R system 1} 
%First we approximate $(\psi_0, \rho_0, \phi_0)$ by smooth functions $(\psi_{0 ,\eta}, \rho_{0, \eta}, \phi_{0, \eta})$ and obtain local approximate solutions \\ $(\psi_\eta, \rho_\eta, \phi_\eta)$ on some time interval $[0, T_\eta]$ due to Theorem \ref{local existence for perturbed system}.\\
Using  part i), we have that for any $T>0$, $(\psi_\eta, \rho_\eta, \phi_\eta)$ is bounded, independently of $\eta$, in the space
\begin{equation}
 \label{space of eta}
  L^\infty((0,T): H^1) \times L^\infty((0,T): L^2) \times L^\infty ((0,T): H^1).
\end{equation}
Therefore, using \eqref{perturbed 2-D Z-R system 1} and Sobolev theorem we infer that $(\partial_t \psi_m, \partial_t \rho_m, \partial_t \phi_m)$ is bounded in 
\[
 L^\infty((0,T): H^{-1}) \times L^\infty((0,T): H^{-1}) \times L^\infty ((0,T): L^2).
\]
Hence, up to a subsequence one can assume that 
\begin{equation}
 \label{weak limit eta}
  \begin{aligned}
   & \psi_m \rightarrow \psi \; \text{in} \; L^\infty((0,T): H^1) - \text{weak*}, \\
   & \rho_m \rightarrow \rho \; \text{in} \; L^\infty ((0,T): L^2) - \text{weak*}, \\
   & \phi_m \rightarrow \phi \; \text{in} \; L^\infty ((0,T): H^1) - \text{weak*}.
  \end{aligned}
\end{equation}
By  Aubin-Lions lemma one can furthermore assume that up to a subsequence 
\begin{equation}
 \label{strong limit eta}
 \psi_m \rightarrow \psi \; \text{in} \; L^p_{loc} ([0,T]: L^q_{loc}(\R^2)),
\end{equation}
for any $2 \leq p, q < \infty$. Similar convergence results hold true for $\phi_m, \rho_m$. 

These convergences allow to pass to the limit in the distribution sense in \eqref{perturbed 2-D Z-R system 1} for $(\psi_m, \rho_m, \phi_m)$, proving that $(\psi, \rho, \phi)$ satisfies \eqref{perturbed 2-D Z-R system 1} in $L^\infty((0,T): H^{-1}) \times L^\infty((0,T): H^{-1}) \times L^\infty ((0,T): L^2)$.\\
The initial condition makes sense since
\[
 (\psi, \rho, \phi) \in C_w([0,T] : H^1) \times C_w([0,T]: L^2) \times C_w([0,T] : H^1)
\]
\end{proof}

\section{Solitary wave solutions in higher dimension}
\vspace{0.5cm}
Let consider now solitary waves solutions of \eqref{2-D Z-R system 1} that is solutions of the form $(e^{i\omega t}\psi(x+\sigma_3t,y),\phi(x,y),\eta(x,y)),\; \omega \in \R,\; \psi\in H^1(\R^2),$ yielding the system

\begin{equation}
 \label{SW-ZR}
  \left \{
   \begin{aligned}
   & -\omega \psi+\delta \psi_{xx}+\sigma_1\nabla_\perp\psi-(\sigma_2- W M^2)|\psi|^2\psi-c W\phi_x \psi=0 \\
   & \Delta\phi+c|\psi|^2_x = 0,
   \end{aligned}
  \right.
\end{equation}

which is similar to the equation for the solitary wave solutions of the elliptic/hyperbolic-elliptic Davey -Stewartson systems in the terminology of \cite{GS}.
By Pohojaev type arguments one obtains (see \cite{GS1} for similar arguments) that  non trivial solutions to \eqref{SW-ZR} cannot exist when $\delta\sigma_1<0.$

On the other hand the existence of non trivial solutions to \eqref{SW-ZR} has been established (\cite{Ci}) in the {\it focusing case}

$$\delta\sigma_1>0,\quad c<0,\quad c(W(\sigma_2-WM^2)<0.$$

\vspace{1cm}
Various stability and instability results of solutions to \eqref{SW-ZR} have been obtained in \cite{Ci2,Ot1,Ot2,Ot3} in the context of the  Davey-Stewartson systems but no similar results seemed to be known when they are viewed as solutions to the Zakharov-Rubenchik systems. In particular one does not know if the solutions of \eqref{SW-ZR} are constrainded minimizers of the Zakharov-Rubenchik system. 

According to the Davey-Stewartson case, one could conjecture that those localized solitary waves are unstable.

\vspace{1cm}

\section{Conclusion and open questions}
We have addressed in this paper some issues on the Zakharov-Rubenchik, Benney-Roskes systems. Many questions remain unsolved for those important systems and we indicate a few below.

1. Justify rigorously the limit of ZR (BR) systems to the  Davey-Stewartson systems. This is non trivial (because of a boundary layer at $t=0$) issue is analogous (but more delicate) to the Schr\"{o}dinger limit of the Zakharov system (see \cite{OT2, OT3}).

2. The present work can be viewed as a preliminary step towards the study of the transverse stability/instability  of the ZR or BR one dimensional dark or bright solitary wave. Perturbations could be localized in (x,y) or periodic in y. The Cauchy problem was addressed in the present paper in both cases.

%One could alternatively think of transverse {\it periodic} perturbations, the functional setting being now %Sobolev spaces of type say $H^s(\R\times \T).$ This would need first a Cauchy theory within this setting. %Note however that since they do not depend on dispersive estimates the results in the present paper on the %Cauchy problem in $\R^2 $ hold true without changes in a $\R\times \T$ setting. 
%\textcolor{red}{HUNG DO YOU AGREE?} \textcolor{blue}{IT IS TRUE IF THE KATO-PONCE ESTIMATE %\eqref{Kato-Ponce estimate} holds in $\R \times \T$.}

In both the functional settings, we plan to come back to those transverse stability  issues  a subsequent work, in the spirit of \cite{RT,RT1,RT2, Mi, MT}. 

3. It is known (\cite{GM, GM1}) that (radially symmetric) solutions of the Zakharov system may blow up in finite time. Such a result is unknown for the Zakharov-Rubenchik, Benney-Roskes system and it would interesting to see if the results in \cite{GM, GM1} extend to \eqref{ZR-system}.

4. The existence result Theorem \ref{local existence for unperturbed system} is established when $\delta \sigma_1>0,$ that is when the 
Schr\"{o}dinger equation in \eqref{ZR-system} is not an "non elliptic" one in the terminology of \cite{GS}. This condition is never  satisfied in the context of purely gravity waves water waves (see \cite{La1} and the discussion above) and it would interesting to relax it. 

5. We recall that an existence result on time scales of order $1/\epsilon$ is needed to fullly justify the Benney-Roskes system. Obtaining such a result is still a challenging open problem.

%%%%%%%%%%%%%%%%%%%%%%%%%%%% 
\vspace{0.3cm}

\section*{Acknowledgements}
We acknowledge financial support by the Austrian Science Foundation FWF under grant No F41 (SFB "VICOM"), grant No F65 (SFB Complexity in PDEs) and grant No W1245 (DK "Nonlinear PDEs"), and by the French Science Foundation ANR, under grant GEODISP).
We acknowledge an anonymous referee for valuable comments and criticisms that helped us to improve the manuscript.

\vspace{0.3cm}

%%%%%%%%%%%%%%%%%%%%%%%%%%%%%%%%%%%%%%%%%%%%%%%%%%%%%%%%%%%%%%%%%%%%%%%%%%%%%%%%%
\bibliographystyle{amsplain}

\end{document}